\documentclass[12pt]{amsart} 

\usepackage{amsmath} \usepackage{amssymb} \usepackage{amsthm}
\usepackage{amscd} 
\def\NN{\Bbb N}%

\def\mm{{\mathfrak{m}}} 
\def\fkp{{\mathfrak{p}}}

\def\opn#1#2{\def#1{\operatorname{#2}}} % to make operators
\opn\chara{char} \opn\length{\ell}  
\opn\embdim{emb\,dim} 
\opn\dim{dim} \opn\codim{codim} \opn\height{height} 
\opn\indeg{indeg} \opn\reg{reg} 
\opn\projdim{proj\,dim} \opn\injdim{inj\,dim} 
\opn\depth{depth}  \opn\grade{grade}
\opn\Ass{Ass} \opn\Min{Min} \opn\Assh{Assh}
\opn\Hom{Hom} \opn\Ker{Ker} \opn\Im{Im} \opn\Coker{Coker} \opn\rank{rank}
\opn\Ann{Ann}
\opn\Spec{Spec} \opn\Proj{Proj}
\opn\Div{div}
\opn\F{F}
\opn\PF{PF}
\opn\t{t}
\opn\g{g}
\opn\e{e}
\opn\SG{SG}
\opn\m{m}
\opn\G{G}
\opn\r{r}
\opn\and{and}
\opn\Card{Card}
\opn\n{n}
\opn\min{min}
\opn\Maximals{Maximals}

\theoremstyle{plain}
\newtheorem{thm}{Theorem}[section] 
\newtheorem{prop}[thm]{Proposition}
\newtheorem{cor}[thm]{Corollary}
\newtheorem{lemma}[thm]{Lemma}

\theoremstyle{definition} 
\newtheorem{defn}[thm]{Definition}
\newtheorem{exam}[thm]{Example} 
\theoremstyle{remark}
\newtheorem{remark}[thm]{Remark}

\newtheorem*{acknowledgement}{Acknowledgment}

\title{Genus of numerical semigroups generated by three elements}
\author{Hirokatsu Nari}
\address{Graduate School of Integrated Basic Sciences, Nihon University, Setagaya-ku, Tokyo, 156-0045, JAPAN}
\email{s6110M09@math.chs.nihon-u.ac.jp}
\author{Takahiro Numata}
\address{Graduate School of Integrated Basic Sciences, Nihon University, Setagaya-ku, Tokyo, 156-0045, JAPAN}
\email{s6110M11@math.chs.nihon-u.ac.jp}
\author{Kei-ichi Watanabe}
\address{Department of Mathematics, College of Humanities and Sciences, 
Nihon University, Setagaya-ku, Tokyo, 156-0045, JAPAN}
\email{watanabe@math.chs.nihon-u.ac.jp}
\date{\today}
\subjclass[2000]{Primary 20M15, Secondary 13F99, 13A02, 16S36}
\keywords{numerical semigroup, pseudo-symmetric semigroup, 
Frobenius number, genus of a semigroup}

\begin{document}

\begin{abstract} Let $H=\left< a,b,c\right>$ be a numerical semigroup generated by three elements and let $R=k[H]$ be its semigroup ring over a field $k$. 
We assume $H$ is not symmetric and assume that the definig ideal of $R$ is 
defined by maximal minors of the matrix $\left( 
\begin{array}{lll}
X^\alpha   &  Y^\beta    &      Z^\gamma  \\
Y^{\beta'} &  Z^{\gamma'}&      X^{\alpha'}
\end{array}
\right)$.  Then we will show that the genus of $H$ is determined by 
the Frobenius number $\F(H)$ and $\alpha\beta\gamma$ or $\alpha '\beta '\gamma '$.  
In particular, we show that $H$ is pseudo-symmetric if and only if $\alpha\beta\gamma=1$ or $\alpha '\beta '\gamma '=1$. \par
Also, we will give a simple algorithm to get all the pseudo-symmetric 
numerical semigroups $H=\left< a,b,c\right>$ with give Frobenius number.
\end{abstract}

\maketitle

\section{Introduction}
 Let $\mathbb{N}$ be the set of nonnegative integers. A \textit{numerical semigroup} 
 $H$ is a subset of $\mathbb{N}$ which is closed under addition
and $\mathbb{N}\setminus H$ is a finite set. We always assume $0\in H$.\par
We define $\F(H):=\max\{n \mid n \not \in H\}, \and \g(H):=\Card(\NN \setminus H)$.
We call $\F(H)$ the {\it Frobenius number} of $H$, and we call $\g(H)$ the {\it genus} of $H$.
Then it is known that $2\g(H)\geq \F(H)+1$.
We denote by $H=\left<a_1,a_2,...,a_n \right>$ the numerical semigroup generated by $a_1,a_2,...,a_n$.
Namely, $H=\sum^{n}_{i=1}a_i \NN$. Moreover, every numerical semigroup admits a unique minimal system of generators.\par
We say that $H$ is {\it symmetric} if $\F(H)$ is odd and for every $a \in \mathbb{Z}$, either
$a\in H$ or $\F(H)-a \in H$, or equivalently, $2\g(H)=\F(H)+1$.
We say that $H$ is {\it pseudo-symmetric} if $\F(H)$ is even and for every $a \in \mathbb{Z}, a\not =\F(H)/2$,
either $a\in H$ or $\F(H)-a \in H$, or equivalently, $2\g(H)=\F(H)+2$.
\par
For a fixed field $k$, a variable $t$ over $k$, let $R=k[H]=k[t^h\;|\; h\in H]$ 
be the semigroup ring of $H$. Then it is known that $H$ semigroup is symmetric (resp. pseudo-symmetric) if and only $R=k[H]$ is a Gorenstein (resp. Kunz) ring 
(see \cite{BDF}).
The $a$-invariant of the semigroup ring $R$  (\cite{GW}) is defined to be 
 $a(R)=\max\{ n \;|\; [H^{1}_{\mm}(R)]_n\ne 0 \}$.
 Since $H^1_{\mm}(R)\cong k[t,t^{-1}]/R$, $a(R)=\max\{m \mid m \not \in H\}$, 
that is, $\F(H)=a(R)$.\par
We say that an integer $x$ is a {\it pseudo-Frobenius number} of $H$ if 
$x \not \in H$ and $x+s \in H$ for all $s \in H, s\ne 0$.
 We denote by $\PF(H)$ the set of pseudo-Frobenius numbers of $H$.
 The cardinality in $\PF(H)$ is called the {\it type} of $H$, denoted by $\t(H)$.
 Since $x\in \PF(H)$ if and only if $t^x$ is in the socle of $H^{1}_{\mm}(k[H])$, 
$\t(H)=\r(k[H])$, the Cohen-Macaulay type of $k[H]$. 
Since $\F(H) \in \PF(H)$, $\t(H)=1$ if and only if $H$ is symmetric.\par
In this paper, we investigate numerical semigroups generated by three elements, 
which is not symmetric.  We put    
 $H= \left<a,b,c \right>$ and always  assume that $H$ is {\it not} symmetric. \par
Let 
We now let $\varphi: S=k[X,Y,Z]\rightarrow R =k[ H ]=k[t^a,t^b,t^c]$ 
the $k$ algebra homomorphism defined by $\varphi (X)=t^a$, $\varphi (Y)=t^b$, 
and $\varphi (Z)=t^c$ and let  $\fkp=\fkp (a,b,c)$ be 
 the kernel of $\varphi$.  Then it is known that if $H$ is not symmetric, 
then the ideal  $\fkp=\Ker(\varphi)$ is 
 generated by the maximal minors of the matrix  
\[\leqno{(1.1)}\qquad
\left( 
\begin{array}{lll}
X^\alpha   &  Y^\beta    &      Z^\gamma  \\
Y^{\beta'} &  Z^{\gamma'}&      X^{\alpha'}
\end{array}
\right)\]
for some positive integers $\alpha,\beta, \gamma, \alpha ', \beta',   \gamma'$ 
(cf. \cite{He}). We want to describe $\g(H)$ by $\alpha,\beta, \gamma, \alpha ', 
\beta',   \gamma'$ and the main goal of this paper is the following Theorem.\\

{\bf Theorem.} Let $H$ be a numerical semigroup as above. Then \par
(1) if $\beta ' b > \alpha a$, then $2\g(H)-(\F(H)+1)=\alpha \beta \gamma$, \par
(2) if $\beta ' b < \alpha a$, then $2\g(H)-(\F(H)+1)=\alpha' \beta' \gamma'$.

As a direct consequence of this Theorem, we can get the characterization of 
pseudo-symmetric semigroups generated by 3 elements.\\

{\bf Corollary.}\quad   Let $H$ be a numerical semigroup as above. Then $H$ is 
pseudo-symmetric if and only if either $\alpha=\beta,=\gamma=1$ or 
$ \alpha '= \beta'=\gamma'=1$.

Also, we will give an  algorithm to classify all pseudo-symmetric 
numerical semigroup $H$ generated by $3$ elements with given Frobenius 
number $\F(H)$. 

\section{Numerical semigroups generated by three elements}

Let $H=\left<a,b,c \right>$ be a numerical semigroup and $R =k[ H ]\cong k[X,Y,Z]/\fkp$ be its semigroup ring over a field $k$.
Then it is known that the ideal $\fkp$ of $S=k[X,Y,Z]$ is generated by the maximal minors of the matrix  
$\left( 
\begin{array}{lll}
X^\alpha   &  Y^\beta    &      Z^\gamma  \\
Y^{\beta'} &  Z^{\gamma'}&      X^{\alpha'}
\end{array}
\right)$,
where $\alpha $, $\beta $, $\gamma $, $\alpha '$, $\beta '$, and $\gamma '$ are positive integers. 
 Since $k[H]/(t^a)\cong k[Y,Z]/(Y^{\beta+\beta'},Y^{\beta'}Z^{\gamma},Z^{\gamma+\gamma'})$,
 the defining ideal of $k[H]/(t^a)$ is generated by the maximal minors of the matrix
 $\left(\begin{array}{lll}
 0 &   Y^\beta    &      Z^\gamma  \\
 Y^{\beta'} &  Z^{\gamma'}& 0
 \end{array}\right)$. 
 Since $a=\dim_kk[H]/(t^a)= \dim_kk[Y,Z]/(Y^{\beta+\beta'},Y^{\beta'}Z^{\gamma},
 Z^{\gamma+\gamma'})$, and likewise for $b,c$, we get the equalitions
 $$\leqno{(2.1.1)}\qquad\begin{array}{lcl}
  a & = & \beta \gamma +\beta '\gamma +\beta '\gamma', \\
b & = & \gamma \alpha +\gamma ' \alpha+\gamma ' \alpha',\\
 c& = & \alpha \beta +\alpha '\beta +\alpha '\beta '.
  \end{array}$$
  
 We put $l=Z^{\gamma +\gamma '}-X^{\alpha'}Y^{\beta}$,
 $m=X^{\alpha +\alpha '}-Y^{\beta '}Z^{\gamma }$, and $n=Y^{\beta +\beta '}-X^{\alpha }Z^{\gamma '}$. There are obvious relations 
 \[X^{\alpha} l+Y^{\beta }m+Z^{\gamma}n=Y^{\beta '}l+Z^{\gamma '}m+X^{\alpha '}n=0.\]
 We put $p=\deg(Z^{\gamma +\gamma '})$, $q=\deg(X^{\alpha +\alpha '})$, $r=\deg(Y^{\beta +\beta '})$, $s=\deg(X^{\alpha })+p$, $t=\deg(Y^{\beta'})+p$.
 Since pd$_S(R)=2$, we get a free resolution of $R$
 \[0\rightarrow S(-s)\oplus S(-t)\rightarrow S(-p)\oplus S(-q)\oplus S(-r)\rightarrow S\rightarrow R\rightarrow 0.\]
 Taking $\Hom_S(*,K_S)=\Hom_S(*,S(-x))$, we get
 \[0\rightarrow S(-x)\rightarrow S(p-x)\oplus S(q-x)\oplus S(r-x)\rightarrow S(s-x)\oplus S(t-x)\rightarrow K_R\rightarrow 0,\]
 where $x=a+b+c$ and $K_R=$Ext$^{2}_{S}(R,K_S)$. 
 
 Since $K_R$ is generated by the elements of degree $- \PF(H)$, 
 from this exact sequence, we have that $\PF(H)=\{s-x,t-x\}$.
 We put $f=s-x$ and $f'=t-x$. \par
 
 By the above argument, we obtain the following results.\par
 
\begin{prop}\label{abc} If $H = \left< a,b,c \right>$ is not symmetric, then \par
 \begin{enumerate}
 \item $(\alpha +\alpha ')a=\beta 'b+\gamma c$ and $\alpha +\alpha'=\min\{n \mid an \in \left<b,c \right>\}$,\par
 \item $(\beta +\beta ')b=\alpha a+\gamma 'c$ and $\beta +\beta '=\min\{n \mid bn \in \left<a,c \right>\}$,\par
 \item $(\gamma +\gamma ')c=\alpha 'a+\beta b$ and $\gamma +\gamma'=\min\{n \mid cn \in \left<a,b \right>\}$.
 \end{enumerate}
 \end{prop}
 \par

 \begin{prop}\label{ff'} If $H = \left< a,b,c \right>$ is not symmetric, then 
$\PF(H)=\{f, f'\}$ where  \par
 \begin{enumerate}
 \item $f=\alpha a+(\gamma +\gamma ')c-(a+b+c)$, \par
 \item $f'=\beta 'b+(\gamma +\gamma ')c-(a+b+c)$.
 \end{enumerate}
 \end{prop}
 
\begin{remark}  Formulas related to our results in this section can be found in 
\cite{RG1}, \cite{RG}.

\end{remark}

\section{Main results}

The following is the key lemma  to prove our main theorem.
\begin{lemma}\label{card} Let $H = \left< a,b,c \right>$ be as in the previous section. 
We assume that $\beta' b > \alpha a$, or equivalently, $f' > f$. Then\par
\begin{enumerate}
\item for $p,q,r \in \NN$, $f'-f +pa+qb+rc\not\in H$ if and only if 
$p<\alpha, q<\beta$ and $r<\gamma$.\par   
\item $\Card\{h \in H \mid f'-f+h \not \in H\}=\alpha \beta \gamma$.\par
\item $\Card[[(f-H)\cap\NN]\setminus (f'-H)] = \alpha \beta \gamma$.
\end{enumerate}
\end{lemma}
\par

\begin{proof} Since $f'-f+\alpha a=b\gamma, f'-f +\beta b= \gamma' c, f'-f +\gamma c 
= \alpha' a\in H$,  $f'-f +pa+qb+rc\in H$ if $p\ge \alpha$ or $q\ge \beta$ or 
$r\ge \gamma$. Conversely, assume $p<\alpha, q<\beta$ and $r<\gamma$ 
and $f'-f +pa+qb+rc= ua+vb+wc\in H$ for some $u,v,w\in \NN$. Then we have 
$(\beta' +q -v)b = (\alpha - p +u) a + (v-r)c$. If $v \ge r$, then this contradicts 
Proposition \ref{abc} (2).  If $r>v$, we have $(\alpha - p +u) a = 
(\beta' +q -v)b  + (r-v)c$.  Then by   Proposition \ref{abc} (1), we must have 
$p-u \ge \alpha'$ and again we have a contradiction since $r-v< \gamma$.   
 This finishes the proof of (1) and (2) is a direct consequence of (1). \par
To show (3), it suffices to note that for $h\in H$, $f-h \not\in f'-H$  
if and only if $f' -(f-h) \not\in H$. \par
Thus we have 
$\Card[(f-H)\setminus (f'-H)] = \Card\{h \in H \mid f'-f+h \not \in H\}
=\alpha \beta \gamma$.  
\end{proof}

\begin{thm}\label{gap} Let $H = \left< a,b,c \right>$ be a numerical semigroup. Then \par
\begin{enumerate}
\item if $\beta ' b > \alpha a$, then $2\g(H)-(\F(H)+1)=\alpha \beta \gamma$, \par
\item if $\beta ' b < \alpha a$, then $2\g(H)-(\F(H)+1)=\alpha' \beta' \gamma'$.
\end{enumerate}
\end{thm}
\par

\begin{proof}
We may assume $\beta ' b > \alpha a$. Then 
by Proposition \ref{ff'}, $\F(H)=f'$. Since $\NN \setminus 
H=((f'-H) \cap \NN) \cup ((f-H) \cap \NN)$, we get 
\[g(H) = \Card [(f'-H)\cap \NN] + \Card[[(f-H)\cap \NN]\setminus (f'-H)]\]
hence by Lemma \ref{card}, 
\[g(H) = (F(H) + 1-g(H)) + \alpha \beta \gamma.\]

\end{proof}
\par
As a corollary, we find a characterization of pseudo-symmetric numerical semigroups
 generated by $3$ elements. \par

\begin{cor}\label{ps} $H$ is pseudo symmetric if and only if \par
\begin{enumerate}
\item if $\beta'b>\alpha a$, then $\alpha = \beta = \gamma = 1$ and \par
\item if $\beta'b<\alpha a$, then $\alpha ' = \beta ' = \gamma ' = 1$.
\end{enumerate}
\end{cor}
\par
\begin{proof}
We may assume  that $\beta'b>\alpha a$. By Theorem \ref{gap}, $2\g(H)-(\F(H)+1)=\alpha \beta \gamma$. 
Since $H$ is pseudo-symmetric if and only if $2\g(H)=\F(H)+2$, we obtain that $\alpha \beta \gamma=1$, or equivalently, $\alpha =\beta =\gamma =1$.
\end{proof}

\section{The structure of a pseudo-symmetric numerical semigroup generated by three elements}

In this section, we assume that $H=\left<a,b,c\right>$ is a 
pseudo-symmetric numerical semigroup.  
Our purpose is to classify, for any fixed fixed even integer $f$,
all the pseudo-symmetric numerical semigroups $H=\left<a,b,c\right>$
with $\F (H) =f$.  For example, it is shown in Exercise 10.8 of  \cite{RG} 
that there is no pseudo-symmetric numerical semigroup $H=\left<a,b,c\right>$
 with $F(H)=12$.  Actually, we can give many examples of such even integer $f$ for which there does not exist numerical semigroup $H=\left<a,b,c\right>$
 with $F(H)=f$.  (It is shown in \cite{RGG} that every even integer is the Frobenius 
 number of some numerical semigroup generated by at most $4$ elements.)   
\par
As is mentioned before, $\fkp=\fkp (a,b,c)$ of $k[X,Y,Z]$ is generated by the maximal minors of the matrix  as in (1.1) 
and by Corollary \ref{ps}, we can always assume  that  $\alpha '=\beta '=\gamma' =1$.  
Recall that in this case we have by (2.1.1),  
\[(4.1.1) \qquad a=\beta \gamma +\gamma +1, \quad b= 
\gamma \alpha + \alpha+1, \quad c=\alpha \beta +\beta +1.\]

The following is the key for our goal.\par

\begin{thm}\label{str} Let $H=\left<a,b,c \right>$ be a pseudo-symmetric
numerical semigroup and assume that 
$\fkp (a,b,c)$  is generated by the maximal minors of the matrix  
$\left( 
\begin{array}{lll}
X^\alpha   &  Y^\beta    &      Z^\gamma  \\
Y &  Z&      X
\end{array}
\right)$.  Then we have 
\[\alpha \beta \gamma =\dfrac{\F(H)}{2}+1.\]
\end{thm}\par

\begin{proof} From our hypothesis and Corollary \ref{ps},  we have $f'<f$. 
Thus  by Proposition \ref{ff'}, 
$\F(H)=f= \alpha a+(\gamma +1)c-(a+b+c)=2\alpha \beta \gamma -2$.
\end{proof}\par

Now, given a positive even integer $f$,  we can list all possibilities of the set 
$\{\alpha, \beta, \gamma\}$ by prime factorization of $\frac{\F(H)}{2}+1$. 

\begin{remark} \label{perm} Let $\sigma$ be a permutation of $\{\alpha, \beta, \gamma\}$. 
 Then it is easy to see that if $\sigma$ is an even permutation, then the set $\{a,b,c\}$ 
 obtained by $\{\sigma(\alpha), \sigma(\beta), \sigma(\gamma)\}$ 
 as in (4.1.1)  is the same and 
 hence the semigroup $H=\left<a,b,c \right>$ does not change. \par
 But if $\sigma$ is an odd permutation, then the set $\{a,b,c\}$ does change.
 So, from the factorization of $\dfrac{\F(H)}{2}+1$, we get $2$ different semigroups
  in general. 
\end{remark}

\begin{exam}\label{ex1}
For example, let us classify all pseudo-symmetric semigroup $H= \left<a,b,c \right>$ with 
$\F(H)=f=18$. Since  we have $\alpha \beta \gamma = f/2 + 1 = 10$ by Theorem \ref{str}, 
we have $\{\alpha, \beta, \gamma\} = \{10, 1, 1\}$ or $\{5, 2, 1\}$.  
But if we put $\{\alpha, \beta, \gamma\} = \{10, 1, 1\}$ in any order to (4.1.1), 
$a,b,c$ are all multiple of $3$ and we don't get a numerical semigroup. \par
Thus we get $2$ semigroups with $\F(H)=18$; if  $(\alpha, \beta, \gamma) = (5, 2,  1)$
 we get $H=\left<4,11,13 \right>$ and if  $(\alpha, \beta, \gamma) = (5, 1,2)$, then 
 we get $H=\left<5,16,7 \right>$. 
\end{exam}

If $f$ is an even integer not divisible by $12$, then there is a pseudo-symmetric semigroup 
$H= \left<a,b,c \right>$ with $\F(H)=f$ by \cite{RGG}.

\begin{prop}\label{ex2}  \cite{RGG} Let $H= \left<a,b,c \right>$ be a numerical semigroup and $\F(H)=f$. Then
\begin{enumerate}
\item If $f$ is an even integer not divisible by $3$, then 
\[\left<3,\frac{f}{2}+3,f+3 \right>\]
is a pseudo-symmetric numerical semigroup with Frobenius number $f$. 
We put $(\alpha, \beta, \gamma) = (f/2+1, 1,1)$. \par
\item If $f$ is a multiple of $6$ and not a multiple of $12$, then if we put 
$(\alpha, \beta, \gamma) = ((f+2)/4, 2,1)$, we get 
\[H=\left<4,\frac{f}{2}+2,\frac{f}{2}+4 \right>, \]
which is  pseudo-symmetric with $\F(H)=f$.
\end{enumerate}
\end{prop}

If $f$ is divisible by $12$, there are many cases such that there does not exist 
pseudo-symmetric semigroup $H= \left<a,b,c \right>$ with $\F(H)=f$. 

\begin{prop}\label{12} We suppose $12 \mid f$. If there exists a 
pseudo-symmetric numerical semigroup 
$H=\left<a,b,c \right>$ with  $\F(H)=f$,
then $f/2+1$ has a prime factor of the form $3k+2 \ (k\geq 1)$.
\end{prop}

\begin{proof}  Otherwise, since $\alpha, \beta, \gamma$ are divisors of 
$f/2+1$, we get  $\alpha\equiv \beta\equiv \gamma\equiv 1$ (mod $3$).
Then by (4.1.1) we see that $a,b,c$ are divisible by $3$ and  
$H=\left<a,b,c \right>$ is not a numerical semigroup.
\end{proof}

\begin{exam}\label{ex3}  Let $f$ be an integer divisible by $12$.
\begin{enumerate}
\item   By Proposition \ref{12}, there is no 
 pseudo-symmetric semigroup $H= \left<a,b,c \right>$ with 
$\F(H)= 12, 24, 36, 60, 72, 84, 96, 120, 132, 
144, 156, 180, 192$.  
\item   On the other hand, there exists
 pseudo-symmetric semigroups $H= \left<a,b,c \right>$ with 
$\F(H)= 48, 168$. Actually,  $H= \left<7,11,31 \right>$ is the  
unique pseudo-symmetric semigroup generated by $3$ elements, 
with $\F(H)=48$ and $H= \left<19,11,103 \right>$ is the  
unique pseudo-symmetric semigroup generated by $3$ elements 
with $\F(H)=168$
\item  The converse of Proposition \ref{12} is not true. Indeed,
If $f=1596$, then $f/2+1=799= 17\times 47$ has a prime factor which is 
congruent to $2$ mod $3$.  But if we substitute $(\alpha, \beta, \gamma)=
(17,47,1)$ (resp. $(47,17,1)$) in (4.1.1), then we get $(a,b,c)=(49,35,847)$
(resp. $(19,95, 817))$.  These are not numerical semigroups since $(a,b,c)$ 
have common prime factor.  It is not difficult to show that $f=1596$ is the smallest 
of such examples. 
\end{enumerate} 
\end{exam}

\section{Simple numerical semigroups}

Let $H$ be a numerical semigroup with minimal system of generators $\{a_1, a_2, ..., a_n\}$. 
We assume that $a_1$ is the least positive integer in $H$.
For every $i \in \{1, ..., n\}$, set
\[\delta _i:=\min \{k \in \mathbb{N} \setminus \{0\} \mid ka_i \in \left<\{a_1, ..., a_n\} \setminus \{a_i\} \right>\}.\]
The notion of simple numerical semigroup was defined in Exercise 10.3 of \cite{RG}.\par

\begin{defn} We say that $H$ is {\it simple} if $a_1=(\delta _2-1)+(\delta _3-1)+\cdot \cdot \cdot +(\delta _n-1)+1$. 
\end{defn}

\begin{prop}\label{sim} Let $H=\left<a_1, a_2, ..., a_n \right>$ be a simple numerical semigroup. Then the type of $H$ is $n-1$. Hence if $H$ is simple with $n\ge 3$, 
then $H$ is not symmetric.
\end{prop}\par

\begin{proof} By definition of pseudo-Frobenius number, we have that 
\[\PF(H)=\{(\delta_2-1)a_2-a_1, (\delta_3-1)a_3-a_1, ..., (\delta_n-1)a_n-a_1\},\]
that is, $H$ has type $n-1$. 
\end{proof}\par

The following is the main result in this section.

\begin{thm}\label{simple-thm} Let $H=\left<a, b, c \right>$ be a numerical semigroup 
defined by the matrix as in (1.1).  If we assume that $a$ is the least positive integer in 
$H$, then $H$ is simple if and only if $\beta '= \gamma=1$.
\end{thm}

\begin{proof} Since $a=\beta \gamma +\beta '\gamma +\beta '\gamma '$, 
and since we have $\delta_2= \beta+\beta ' , \delta_3=\gamma +\gamma$,
$H$ is simple if and only if 
\[ \beta \gamma +\beta '\gamma +\beta '\gamma ' = 
\beta+\beta '+ \gamma +\gamma' -1\]
or, equivalently,
\[(\beta -1)(\gamma -1)+(\beta'-1)(\gamma'-1)+(\beta'\gamma-1)=0.\]
Since $\beta, \beta',\gamma, \gamma'$ are positive integers, 
the latter equation is equivalent to $\beta'=\gamma=1$.
\end{proof}

\begin{acknowledgement}
The third named author is partially supported by Grant-in-Aid for Scientific 
Research 20540050 and Individual Research Expense of College of Humanity and Sciences, Nihon University.
\end{acknowledgement}

\end{document}